\documentclass{article}

\usepackage{arxiv}

\usepackage[utf8]{inputenc} % allow utf-8 input
\usepackage[T1]{fontenc}    % use 8-bit T1 fonts
\usepackage{hyperref}       % hyperlinks
\usepackage{url}            % simple URL typesetting
\usepackage{booktabs}       % professional-quality tables
\usepackage{amsfonts}       % blackboard math symbols
\usepackage{nicefrac}       % compact symbols for 1/2, etc.
\usepackage{microtype}      % microtypography
\usepackage{lipsum}
\usepackage{graphicx}
\usepackage{pgf,tikz}
\usetikzlibrary{arrows}
\usepackage{subfigure}
\usepackage{amsmath}
\usepackage{amsthm}
\newcommand{\mat}[1]{\mbox{\boldmath{$#1$}}}

\newtheorem{thm}{Theorem}

\newtheorem{cor}{Corollary}

\title{B$_{1}$-EPG representations using block-cutpoint trees}

\author{
  Vitor Tocci F. de Luca\\
  Universidade do Estado do Rio de Janeiro\\
  Rio de Janeiro, Brazil\\
  \texttt{toccivitor8@gmail.com} \\
   \And
 Fabiano de S. Oliveira \\
  Universidade do Estado do Rio de Janeiro\\
  Rio de Janeiro, Brazil\\
  \texttt{fabiano.oliveira@ime.uerj.br} \\
  \AND
  Jayme L. Szwarcfiter \\
  Universidade Federal do Rio de Janeiro \hspace{0.5cm} Universidade do Estado do Rio de Janeiro \\
  Rio de Janeiro, Brazil \\
  \texttt{jayme@nce.ufrj.br} \\
}

\begin{document}
\maketitle

\begin{abstract}
In this paper, we are interested in the edge intersection graphs of paths of a grid where each path has at most one bend, called B$_{1}$-EPG graphs and first introduced by Golumbic et al (2009). We also consider a proper subclass of B$_{1}$-EPG, the $\llcorner$-EPG graphs, which allows paths only in ``$\llcorner$'' shape. We show that two superclasses of trees are B$_1$-EPG (one of them being the cactus graphs). On the other hand, we show that the block graphs are $\llcorner$-EPG and provide a linear time algorithm to produce $\llcorner$-EPG representations of generalization of trees. These proofs employed a new technique from previous results in the area based on block-cutpoint trees of the respective graphs.
\end{abstract}

% keywords can be removed

\section{Introduction}\label{intro}
Let $\mathcal{P}$ be a family of nontrivial paths on a rectangular grid $\mathcal{G}$. We define the \emph{edge intersection graph} EPG($\mathcal{P}$) of $\mathcal{P}$ as the graph whose vertex set is $\mathcal{P}$ and such that $(P, Q)$ is an edge of EPG($\mathcal{P}$) if and only if paths $P$ and $Q$ share at least one grid edge of $\mathcal{G}$. A graph $G$ is called an \emph{edge intersection graph of paths on a grid (EPG)} if $G = EPG(\mathcal{P})$ for some family of paths $\mathcal{P}$ on a grid $\mathcal{G}$, and $\mathcal{P}$ is an \emph{EPG representation} of $G$. EPG graphs were first introduced by Golumbic et al in \cite{golumbic2009edge} motivated from circuit layout problems \cite{Brady}. Figure~\ref{fig: example} illustrates the EPG-graph corresponding to the family of paths presented in the figure.
\begin{figure}[htb]
    \centering
    \includegraphics[scale=1]{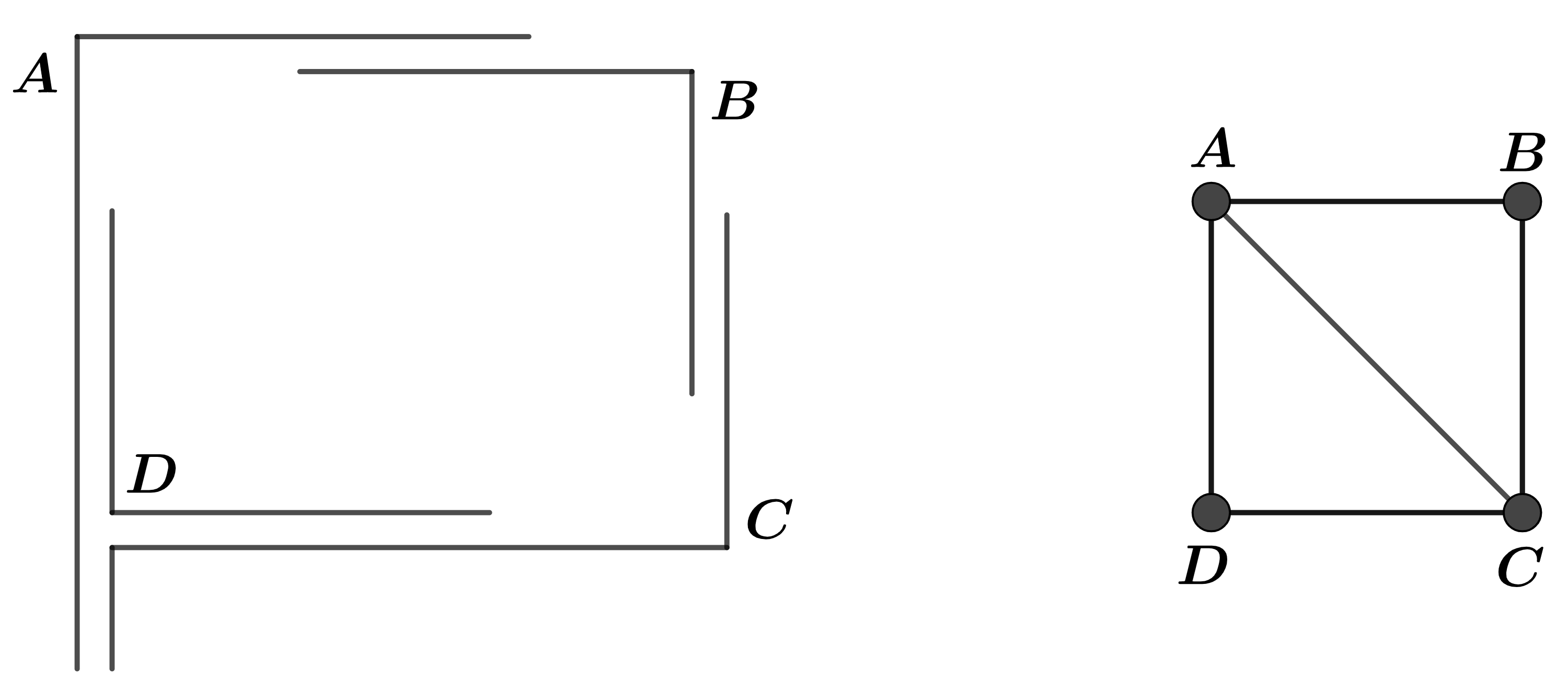}
    \caption{A B$_{2}$-EPG representation $\mathcal{P}$ and its corresponding EPG graph EPG($\mathcal{P}$).}
    \label{fig: example}
\end{figure}

A turn of a path at a grid point is called a \emph{bend} and the grid point in which a bend occurs is called a \emph{bend point}. An EPG representation is a \emph{B$_{k}$-EPG representation} if each path has at most $k$ bends. A graph that has a B$_{k}$-EPG representation is called B$_{k}$-EPG. Therefore, the graph defined in Figure~\ref{fig: example} is B$_{2}$-EPG, as the representation shows. However, it is possible to show that there is a B$_{1}$-EPG representation of $G$ and, thus, $G$ is also B$_{1}$-EPG. The time complexity of recognizing B$_{k}$-EPG is polynomial for $k = 0$ \cite{golumbic2009edge}, and NP-hard for $k=1$ \cite{Heldt_bend} and $k=2$ \cite{Pergel}, whereas is unknown for other values of $k$.

A \emph{block} $B$ of a graph $G$ is a maximal biconnected subgraph of $G$. A vertex $v$ of a connected graph $G$ is a \emph{cut vertex} if $G-v$ is disconnected. For a graph $G$, we define its \emph{block-cutpoint tree} \cite{harary1969graph} (BC-tree) $T$ as follows. There is a vertex in $T$ corresponding each block of $G$, called a \emph{block vertex}, and a vertex for each cut vertex of $G$, called as such in $T$. A cut vertex $c$ forms an edge with a block vertex $b$ if the block corresponding to $b$ contains $c$ in $G$. The only existing vertices and edges of $T$ are those previously described. Figure~\ref{fig: BC-tree example} depicts a graph and its respective BC-tree.
\begin{figure}[htb]

\center
\subfigure[][]{\includegraphics[scale=0.8]{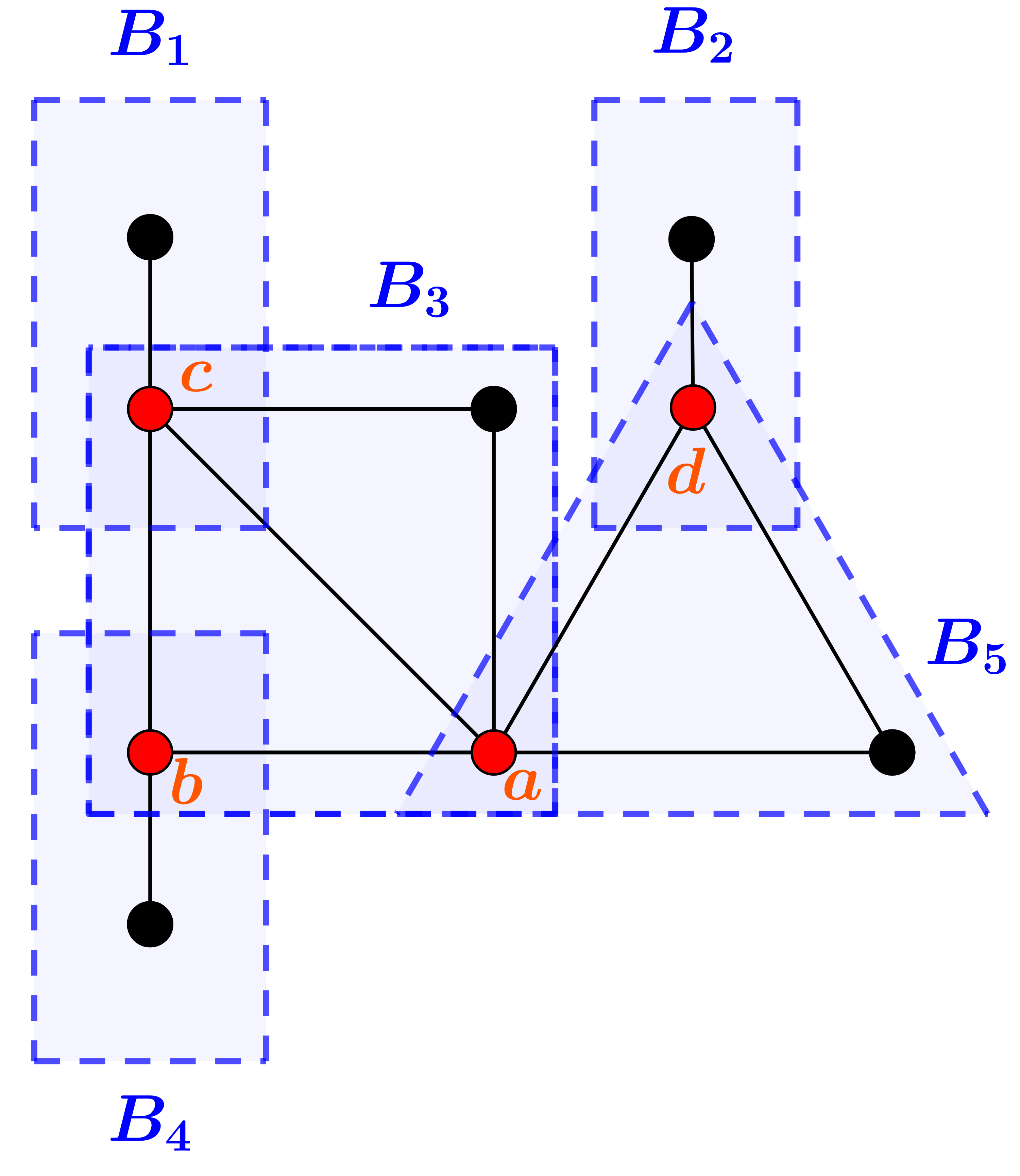} \label{fig: example2}}
\qquad
\subfigure[][]{\includegraphics[scale=0.8]{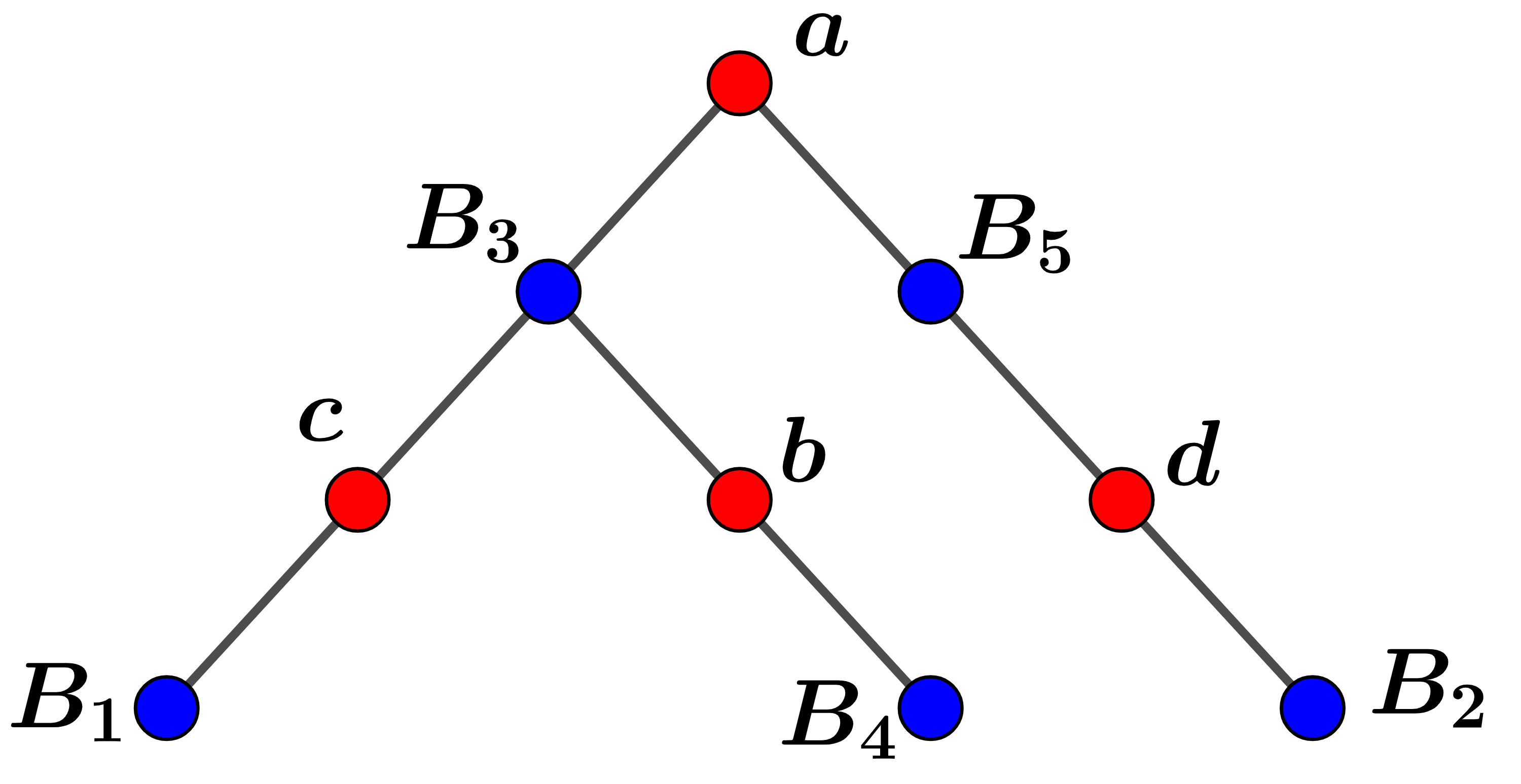} \label{fig: example2_BCtree}}
\caption{A graph and its respective BC-tree. The cut vertices are marked in red.}
\label{fig: BC-tree example}

\end{figure}

A \emph{universal vertex} is a vertex of $G$ that is adjacent to all other vertices of $G$. For $X\subseteq V(G)$, we denote by $G[X]$ the subgraph induced by $X$. A cycle with $k$ vertices is denoted by $C_k$.

In B$_{1}$-EPG representations, each path has one of the following shapes: $\llcorner, \ulcorner, \lrcorner, \urcorner$, besides horizontal or vertical segments. One may consider more restrictive subclasses of B$_{1}$-EPG by limiting the type of bends allowed in the representation. This arises the definition of ``x''-EPG graphs, where ``x'' stands for a sequence of path shapes allowed in the class. For example, the $\llcorner\urcorner$-EPG graphs are those in which only the ``$\llcorner$" or the ``$\urcorner$" shapes are allowed. Although that might imply the study of $2^4$ different subclasses, corresponding to all subsets of $\{\llcorner, \ulcorner, \lrcorner, \urcorner\}$, only the $\llcorner$-EPG, $\lrcorner\llcorner$-EPG, $\llcorner\urcorner$-EPG and $\llcorner\ulcorner\urcorner$-EPG may be considered, since all others do not define distinct subclasses (their representations are isomorphic to these four up to 90 degree rotations and reflections).

\section{A B$_1$-EPG representation of a superclass of trees}

In this section, we describe a B$_1$-EPG representation of a superclass of trees, inspired on the representation of trees described in \cite{golumbic2009edge}. The novelty of the following results are the usage of BC-trees to obtain EPG representations.

\begin{thm} \label{theorem1}
Let $G$ be a graph such that every block of $G$ is B$_{1}$-EPG and every cut vertex $v$ of $G$ is a universal vertex in the blocks of $G$ in which $v$ is contained. Then, $G$ is B$_1$-EPG.
\end{thm}
\begin{proof}
The result is trivial if $G$ does not have cut vertices, since $G$ consists of a single block. Therefore, we assume from now on that there is a cut vertex in $G$. The theorem is proved by induction. Actually, we prove a stronger claim, stated as follows: given any graph $G$ satisfying the theorem conditions and a BC-tree $T$ of $G$ rooted at some cut vertex $r$, there exists a B$_1$-EPG representation $\mathcal{R} = \{P_v \mid v \in V(G) \}$ of $G$ in which:
\begin{itemize}
\item [(i)] $P_r$ is a vertical path with no bends in $\mathcal{R}$;
\item [(ii)] all paths but $P_r$ are constrained within the horizontal portion of the grid defined by $P_r$ and at the right of it.
\end{itemize}

Let $B_1, B_2, \ldots, B_t$ be the block vertices which are children of $r$ and let $T_{i1}, T_{i2}, \ldots, T_{ij_{i}}$ be the subtrees rooted at $B_i$, for all $1\leq i\leq t$ (see Figure~\ref{fig: BC-tree}). The leaves of $T$ are the blocks of $G$ having exactly one cut vertex.
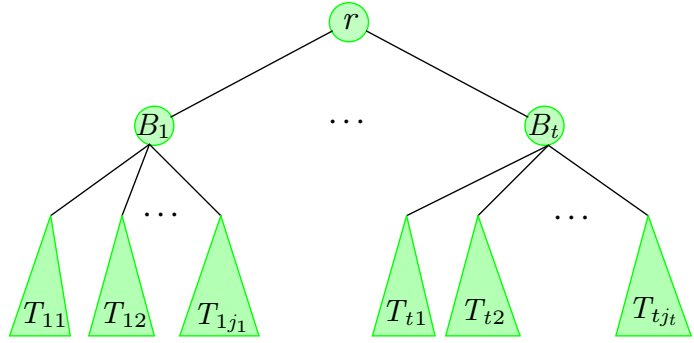
\begin{figure}[htb]
    \centering
    \definecolor{qqffqq}{rgb}{0,1,0}
\scalebox{1.3}{
\begin{tikzpicture}[line cap=round,line join=round,>=triangle 45,x=1.0cm,y=1.0cm]
\clip(1.76,1.92) rectangle (8.99,5.62);
\draw [color=qqffqq,fill=qqffqq,fill opacity=0.25] (5.32,5.32) circle (0.2cm);
\draw [color=qqffqq,fill=qqffqq,fill opacity=0.25] (3.33,4.26) circle (0.2cm);
\draw [color=qqffqq,fill=qqffqq,fill opacity=0.25] (7.32,4.26) circle (0.2cm);
\fill[color=qqffqq,fill=qqffqq,fill opacity=0.3] (1.85,2.11) -- (2.27,3.34) -- (2.47,2.11) -- cycle;
\fill[color=qqffqq,fill=qqffqq,fill opacity=0.3] (2.66,2.11) -- (3,3.34) -- (3.33,2.11) -- cycle;
\fill[color=qqffqq,fill=qqffqq,fill opacity=0.3] (3.59,2.11) -- (4.01,3.34) -- (4.4,2.11) -- cycle;
\fill[color=qqffqq,fill=qqffqq,fill opacity=0.3] (5.56,2.11) -- (5.91,3.34) -- (6.2,2.11) -- cycle;
\fill[color=qqffqq,fill=qqffqq,fill opacity=0.3] (6.31,2.11) -- (6.64,3.34) -- (7.07,2.11) -- cycle;
\fill[color=qqffqq,fill=qqffqq,fill opacity=0.3] (8.05,2.11) -- (8.38,3.34) -- (8.82,2.11) -- cycle;
\draw (3.28,4.07)-- (2.27,3.34);
\draw (3.28,4.07)-- (3,3.34);
\draw (3.28,4.07)-- (4.01,3.34);
\draw (7.36,4.06)-- (5.91,3.34);
\draw (7.36,4.06)-- (6.64,3.34);
\draw (7.36,4.06)-- (8.38,3.34);
\draw [color=qqffqq] (1.85,2.11)-- (2.27,3.34);
\draw [color=qqffqq] (2.27,3.34)-- (2.47,2.11);
\draw [color=qqffqq] (2.47,2.11)-- (1.85,2.11);
\draw [color=qqffqq] (2.66,2.11)-- (3,3.34);
\draw [color=qqffqq] (3,3.34)-- (3.33,2.11);
\draw [color=qqffqq] (3.33,2.11)-- (2.66,2.11);
\draw [color=qqffqq] (3.59,2.11)-- (4.01,3.34);
\draw [color=qqffqq] (4.01,3.34)-- (4.4,2.11);
\draw [color=qqffqq] (4.4,2.11)-- (3.59,2.11);
\draw [color=qqffqq] (5.56,2.11)-- (5.91,3.34);
\draw [color=qqffqq] (5.91,3.34)-- (6.2,2.11);
\draw [color=qqffqq] (6.2,2.11)-- (5.56,2.11);
\draw [color=qqffqq] (6.31,2.11)-- (6.64,3.34);
\draw [color=qqffqq] (6.64,3.34)-- (7.07,2.11);
\draw [color=qqffqq] (7.07,2.11)-- (6.31,2.11);
\draw [color=qqffqq] (8.05,2.11)-- (8.38,3.34);
\draw [color=qqffqq] (8.38,3.34)-- (8.82,2.11);
\draw [color=qqffqq] (8.82,2.11)-- (8.05,2.11);
\draw (5.15,5.23)-- (3.5,4.35);
\draw (5.5,5.23)-- (7.15,4.35);
\draw (4.9,4.45) node[anchor=north west] {$$ \ldots $$};
\draw (3,3.5) node[anchor=north west] {$$ \ldots $$};
\draw (7.2,3.47) node[anchor=north west] {$$ \ldots $$};
\draw (5.12,5.53) node[anchor=north west] {$ \mathit{r} $};
\draw (3.01,4.51) node[anchor=north west] {\small{$\mathit{B}_1 $}};
\draw (7.02,4.51) node[anchor=north west] {\small{$\mathbf{ \mathit{B_t} }$}};
\draw (1.82,2.58) node[anchor=north west] {\small{$ \mathit{T}_{11} $}};
\draw (2.64,2.58) node[anchor=north west] {\small{$ \mathit{T}_{12} $}};
\draw (3.6,2.57) node[anchor=north west] {\small{$ \mathit{T}_{1\mathit{j}_{1}} $}};
\draw (5.53,2.6) node[anchor=north west] {\small{$ \mathit{T}_{\mathit{t}1} $}};
\draw (6.33,2.61) node[anchor=north west] {\small{$ \mathit{T}_{\mathit{t}2} $}};
\draw (8.05,2.63) node[anchor=north west] {\small{$ \mathit{T}_{t\mathit{j_{t}}} $}};
\end{tikzpicture}}
    \caption{The rooted BC-tree $T$ of a graph.}
    \label{fig: BC-tree}
\end{figure}
From $T$, build the representation $\mathcal{R}$ of $G$ as follows. First, build an arbitrary vertical path $P_r$ in the grid $\mathcal{G}$, corresponding the root $r$. Next, divide the vertical portion of  $\mathcal{G}$ defined by $P_r$ and at the right of it into $t$ vertical subgrids, $\mathcal{G}_{1}, \mathcal{G}_{2},\ldots,\mathcal{G}_{t}$, with a row space between them such that the $i$-th subgrid will contain the paths corresponding to the cut vertices that are descendants of $B_i$ in $T$.
So, each subgrid $\mathcal{G}_{i}$ is constructed as follows. We first represent the children of $B_{i}$ as disjoint $\llcorner$-shaped paths, all sharing the same grid column in which $P_{r}$ lies, since by the hypothesis, the children of $B_{i}$ are all adjacent to $r$. Now, for each $B_i$, we build the following paths:
\begin{itemize}
    \item[-] those corresponding to vertices of $B_i$ that are not cut vertices of $G$ (as those vertices in black in Figure~\ref{fig: example2}); let us call the set of such vertices as $B'_i$;
    \item[-] those belonging to the induced subgraphs of $G$ corresponding to the BC-trees $T_{i1}, T_{i2}, \ldots, T_{ij_{i}}$.
\end{itemize}
These paths will be placed on the marked regions of $G_i$ of Figure~\ref{fig:figureA}.
\begin{figure}[htb]
    \centering
    \includegraphics[scale=1.2]{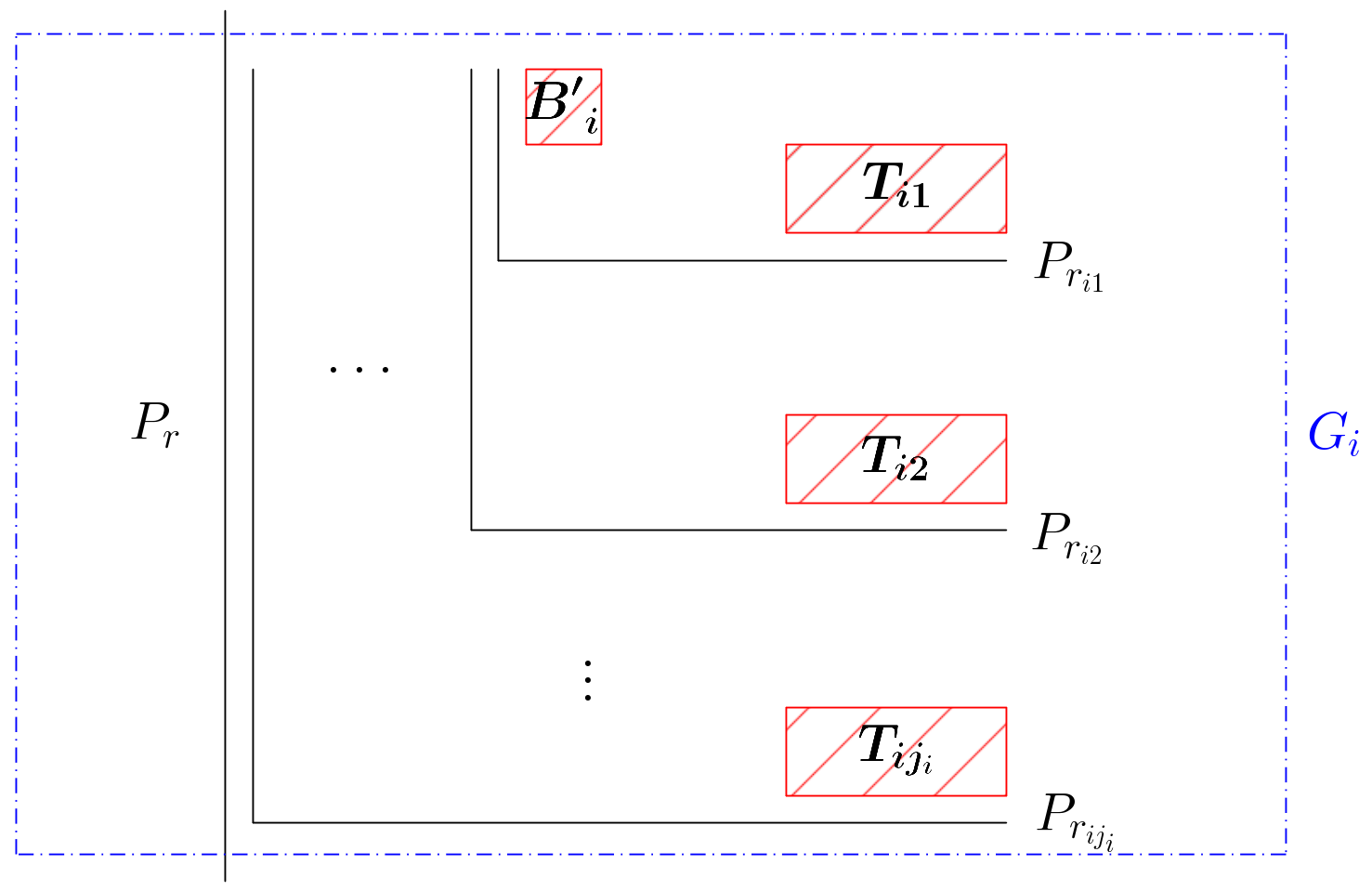}
    \caption{A subgrid $G_i$.}
    \label{fig:figureA}
\end{figure}
So, it remains to define how the paths belonging to the regions $B'_i$ and $T_{ij}$ will be build, for all $1 \leq j \leq j_i$.

So, by the claim hypothesis, $r$ is universal to $B_i$ and $B_i$ is a B$_1$-EPG graph. Therefore, let $\mathcal{R}'$ be a B$_1$-EPG representation of $B_i$. Without loss of generality, considering the operation of rotating the representation, let $P'_r$ be an $\llcorner$-path corresponding $r$ in $\mathcal{R}'$ and let $p$ be its bend point. Since $r$ is universal to $B_i$, all other paths must share a grid edge with $P'_r$. Transform $\mathcal{R}'$ in the following way:

\begin{itemize}
    \item[-] For all $P'_z$, a path of $\mathcal{R}'$ that intersects all other paths of $\mathcal{R}'$ and is not coincident to $P'_r$, modify $P'_z$ by making it coincident to $P'_r$.
    
    \item[-] For all $P'_z$, a path of $\mathcal{R}'$ which contains $p$ and the grid point immediately below of $p$, modify $P'_z$ by removing the part of the path that goes from $p$ downwards (that is, making $p$ an endpoint of $P'_z$). Such a modification does not change the intersections of $P'_z$. Clearly, by construction, it does not increase the intersections. To see that it does not decrease as well, note that if $P'_z$ lost an edge intersection to some path $P'_w$, it is because $P'_w$ would intersect $P'_z$ only in the ``leg'' that was removed, which would imply that $P'_w$ does not intersect $P'_r$, an absurd.
 
    \item[-] For all $P'_z$, a path of $\mathcal{R}'$ which contains $p$ and the grid point immediately at the left of $p$, modify $P'_z$ in an analogous way, removing the part of the paths that are to the left of $p$.
\end{itemize}

 Finally, $\mathcal{R}'$ can be transformed such that all universal vertices become vertical paths, by ``unbending'' them at the grid point $p$ (see Figure~\ref{fig: Subcases}). 

\begin{figure}[htb]

\center
\includegraphics[scale=1]{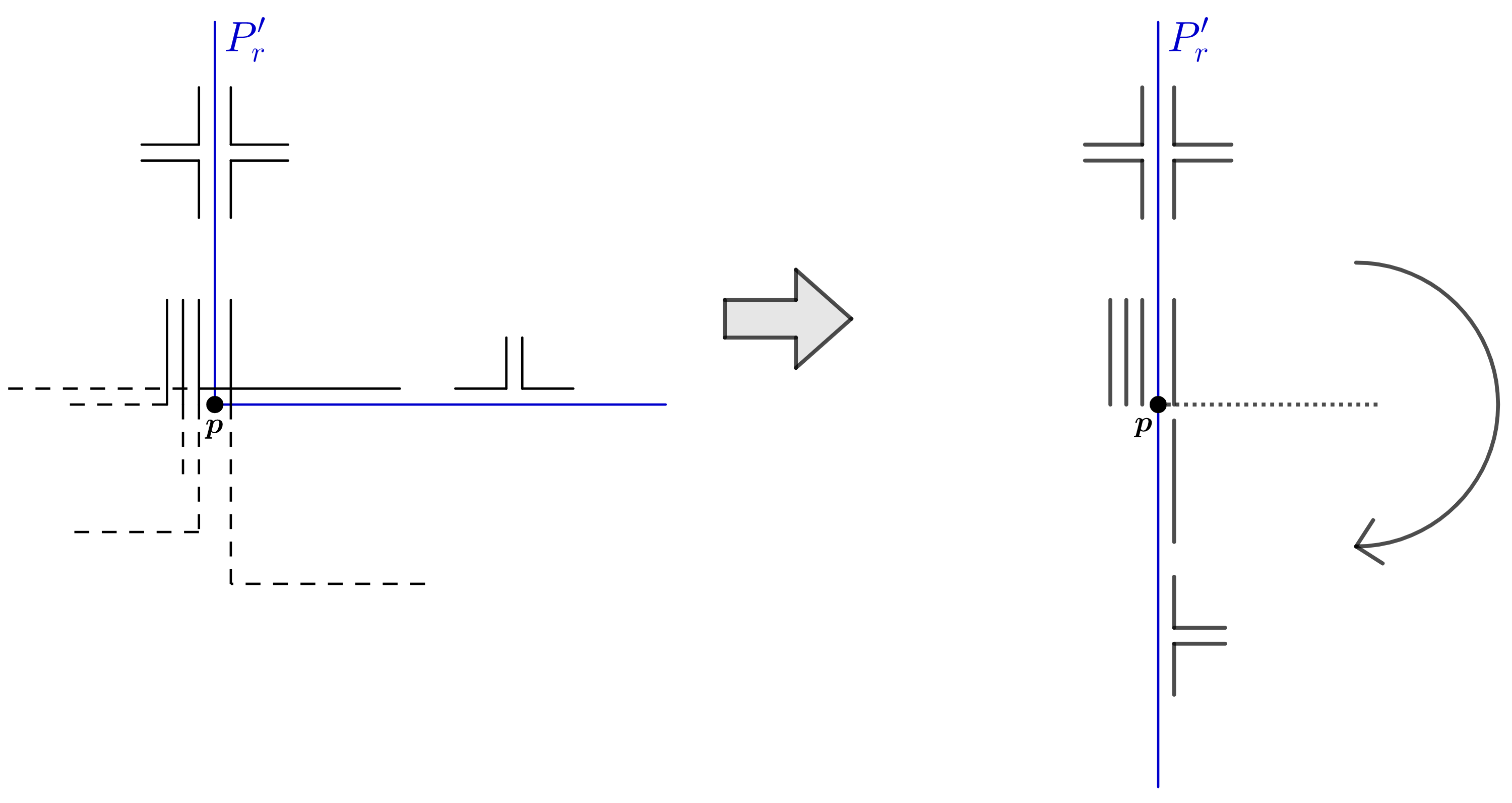}
\caption{Transformations of $P'_z$.}
\label{fig: Subcases}

\end{figure}

%As $G[B_{i}]$ are all interval graphs, we can then represent the non-cut vertices of $B_{i}$ as vertical paths sharing the same column as the paths representing the children of $B_{i}$.

For the $T_{ij}$ portion of the representation, let $r_{ij}$ be the root of $T_{ij}$. Applying induction hypothesis, we obtain B$_1$-EPG representations of each subtree that have vertical paths representing each root and the entire representation is bounded as described previously in (ii). Thus, we can clearly attach each one of the representations to its respective portion of the model being built, rotated 90 degrees in counter-clockwise (see Figure~\ref{fig: B1-EPG BC-tree 2}). This concludes the proof.
\begin{figure}[htb]

\center
\subfigure[][]{\includegraphics[scale=1]{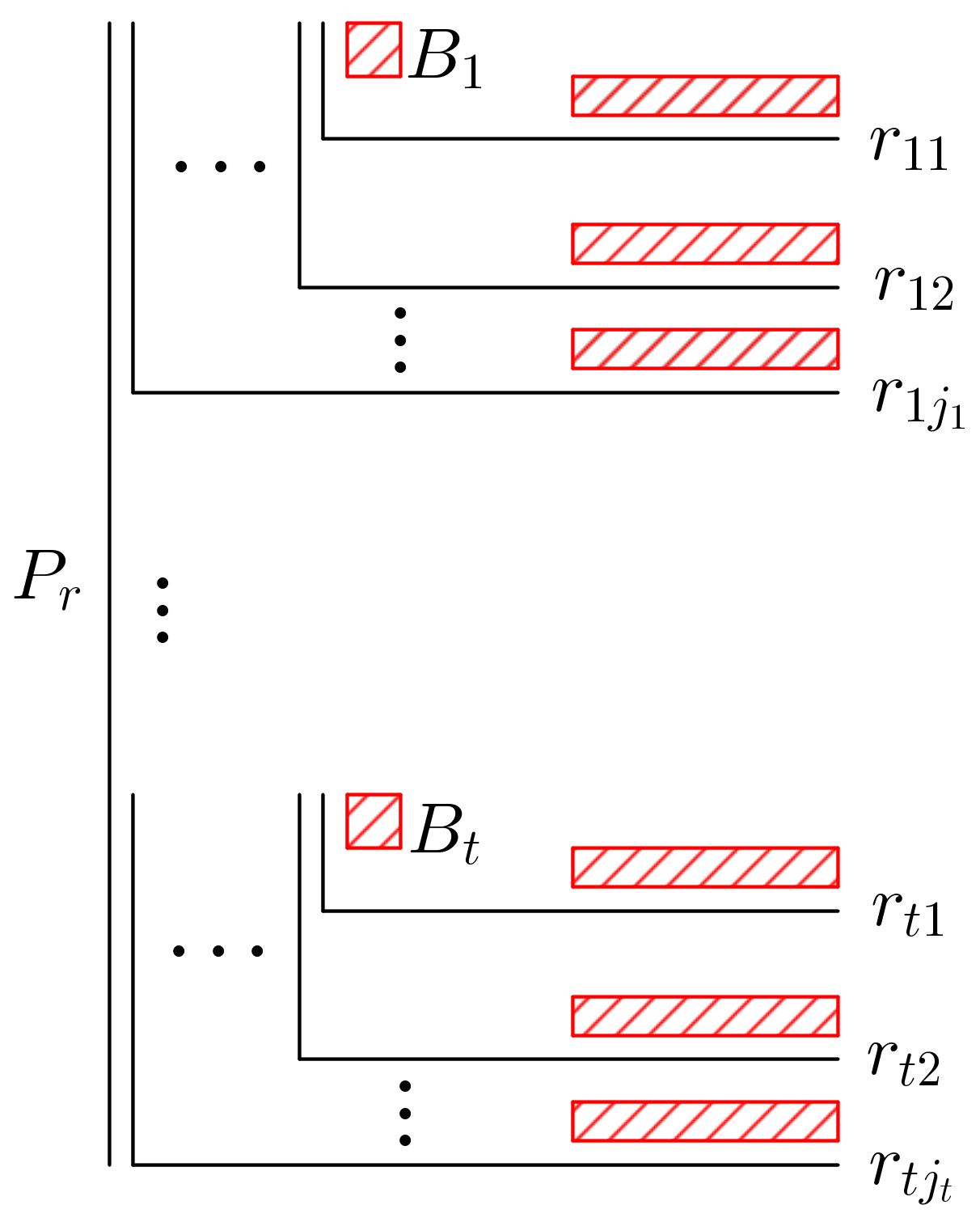} \label{fig: desenho}}
\qquad
\subfigure[][]{\includegraphics[scale=1.5]{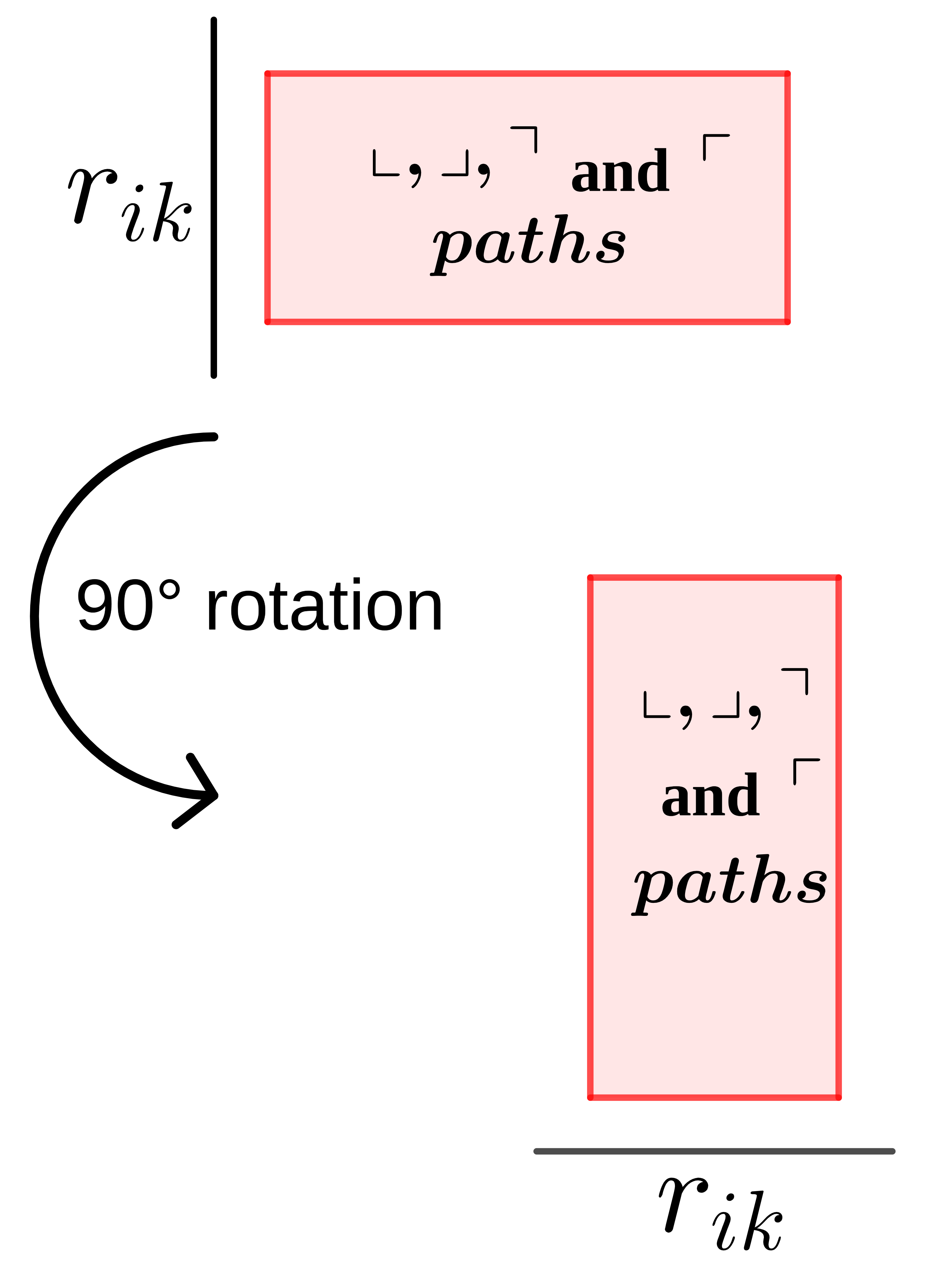} \label{fig: desenho2_1}}
\caption{B$_1$-EPG representation of $G$ after induction step.}
\label{fig: B1-EPG BC-tree 2}

\end{figure}
\end{proof}

It is assumed in Theorem~\ref{theorem1} that the blocks of $G$ are B$_1$-EPG. If, instead, it is assumed that the blocks of $G$ are $\llcorner$-EPG, we get an $\llcorner$-EPG representation of a superclass of trees. The representation to be shown yields a distinct $\llcorner$-representation of trees of the one described in \cite{golumbic2009edge}.
\begin{thm} \label{theorem2}
Let $G$ be a graph such that every block of $G$ is $\llcorner$-EPG and every cut vertex $v$ of $G$ is a universal vertex in the blocks of $G$ in which $v$ is contained. Then, $G$ is $\llcorner$-EPG.
\end{thm}
\begin{proof}
The theorem is proved by induction. Once again, a stronger claim is actually proved, stated as follows: given any graph $G$ satisfying the theorem conditions and a BC-tree $T$ of $G$ rooted at some cut vertex $r$, there exists an $\llcorner$-EPG representation $\mathcal{R} = \{P_v \mid v \in V(G) \}$ of $G$ in which:
\begin{itemize}
\item [(i)] $P_r$ is a vertical path with no bends in $\mathcal{R}$;
\item [(ii)] all paths but $P_r$ are constrained within the horizontal portion of the grid defined by $P_r$ and at the right of it.
\end{itemize}

Let $B_1, B_2, \ldots, B_t$ be the block vertices which are children of $r$ and let $T_{i1}, T_{i2}, \ldots, T_{ij_{i}}$ be the subtrees rooted at $B_i$, for all $1\leq i\leq t$ (see Figure~\ref{fig: BC-tree}). The leaves of $T$ are the blocks of $G$ having exactly one cut vertex.

From $T$, build the representation $\mathcal{R}$ of $G$ as follows. First, build an arbitrary vertical path $P_r$ in the grid $\mathcal{G}$, corresponding the root $r$. Next, divide the vertical portion of  $\mathcal{G}$ defined by $P_r$ and at the right of it into $t$ vertical subgrids, $\mathcal{G}_{1}, \mathcal{G}_{2},\ldots,\mathcal{G}_{t}$, with a row space between them such that the $i$-th subgrid will contain the paths corresponding to the cut vertices that are descendants of $B_i$ in $T$.
So, each subgrid $\mathcal{G}_{i}$ is constructed as follows. We first represent the children of $B_{i}$ as disjoint $\llcorner$-shaped paths, all sharing the same grid column in which $P_{r}$ lies, since by the hypothesis, the children of $B_{i}$ are all adjacent to $r$. Now, for each $B_i$, we build the following paths:
\begin{itemize}
    \item[-] those corresponding to vertices of $B_i$ that are not cut vertices of $G$; let us call the set of such vertices as $B'_i$;
    \item[-] those belonging to the induced subgraphs of $G$ corresponding to the BC-trees $T_{i1}, T_{i2}, \ldots, T_{ij_{i}}$.
\end{itemize}
These paths will be placed on the marked regions of $G_i$ of Figure~\ref{fig:figureA}.
So, it remains to define how the paths belonging to the regions $B'_i$ and $T_{ij}$ will be build, for all $1 \leq j \leq j_i$.

So, by the claim hypothesis, $r$ is universal to $B_i$ and $B_i$ is a B$_1$-EPG graph. Therefore, let $\mathcal{R}'$ be an $\llcorner$-EPG representation of $B_i$. Without loss of generality, considering the operation of rotating the representation, let $P'_r$ be an $\llcorner$-path corresponding $r$ in $\mathcal{R}'$ and let $p$ be its bend point. Since $r$ is universal to $B_i$, all other paths must share a grid edge with $P'_r$. Transform $\mathcal{R}'$ in the following way:
\begin{itemize}
    \item[-] For all $P'_z$, an $\llcorner$-path of $\mathcal{R}'$ that intersects $P'_r$ only in its vertical (resp. horizontal) portion, shorten the horizontal (resp. vertical) ``leg'' of $P'_z$ until that ``leg'' is degenerated into a single point, that is, transform $P'_z$ into a line segment by removing the horizontal (resp. vertical) ``leg'' of the path.

    \item[-] For all $P'_z$, a path of $\mathcal{R}'$ which contains $p$ and the grid point immediately below of $p$, modify $P'_z$ by removing the part of the path that goes from $p$ downwards (that is, making $p$ an endpoint of $P'_z$). Such a modification does not change the intersections of $P'_z$. Clearly, by construction, it does not increase the intersections. To see that it does not decrease as well, note that if $P'_z$ lost an edge intersection to some path $P'_w$, it is because $P'_w$ would intersect $P'_z$ only in the ``leg'' that was removed, which would imply that $P'_w$ does not intersect $P'_r$, an absurd.
 
    \item[-] For all $P'_z$, a path of $\mathcal{R}'$ which contains $p$ and the grid point immediately at the left of $p$, modify $P'_z$ in an analogous way, removing the part of the paths that are to the left of $p$.
\end{itemize}

Finally, $\mathcal{R}'$ can be transformed into an interval model such all universal vertices become vertical paths, by ``unbending'' them at the grid point $p$ (see Figure~\ref{fig: Subcases_2}).

\begin{figure}[htb]

\center
\includegraphics[scale=1]{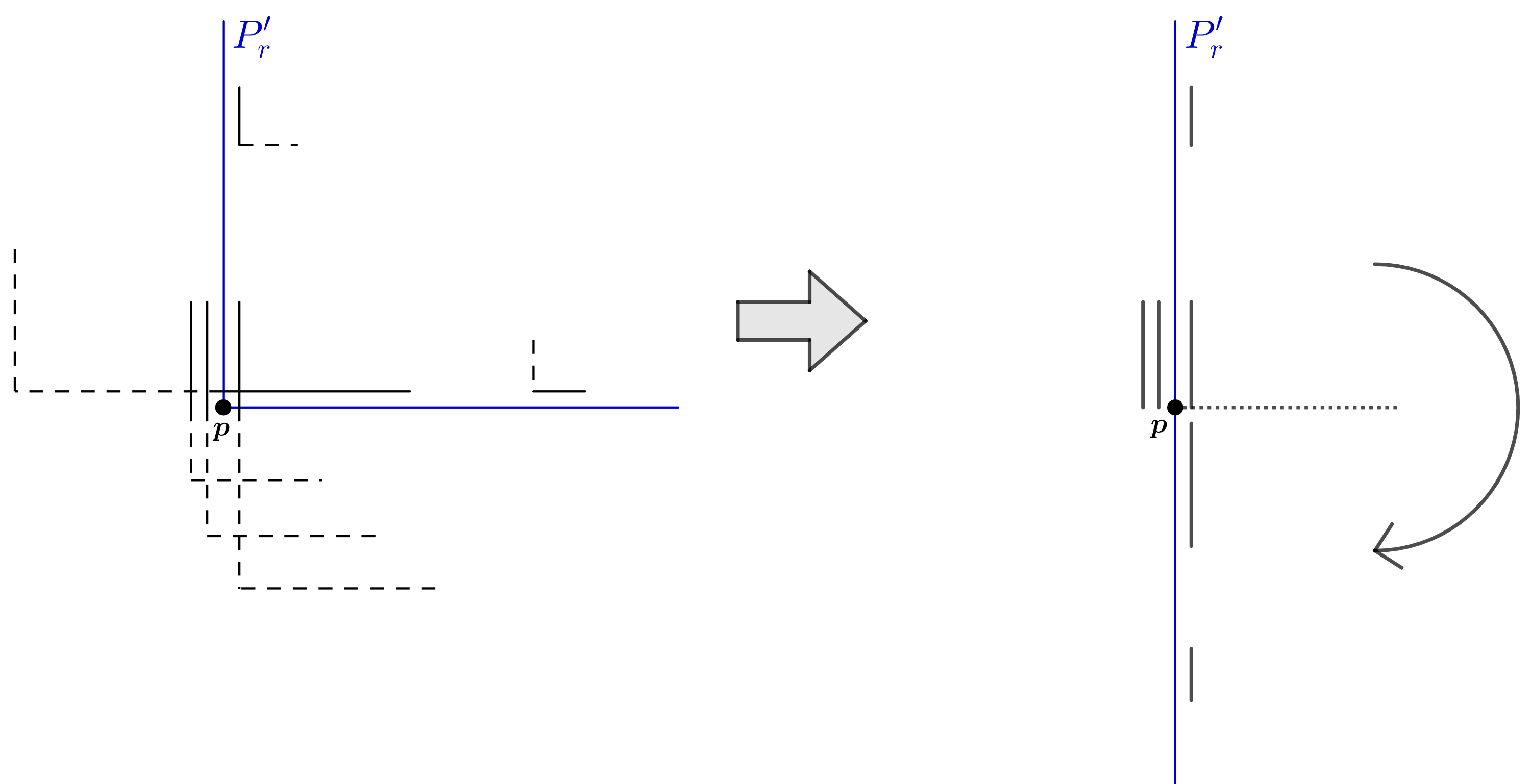}
\caption{Transformations of $P'_z$.}
\label{fig: Subcases_2}

\end{figure}

For the $T_{ij}$ portion of the representation, let $r_{ij}$ be the root of $T_{ij}$. Applying induction hypothesis, we obtain $\llcorner$-EPG representations of each subtree that have vertical paths representing each root and the entire representation is bounded as described previously in (ii). Thus, we can clearly attach each one of the representations to its respective portion of the model being built, rotated 90 degrees in counter-clockwise. Note that due to the rotation of the $\llcorner$-EPG representations, the paths are transformed into $\lrcorner$-shaped paths. In order to obtain an $\llcorner$-EPG representation of $G$, flip the representations of each subtree horizontally (see Figures~\ref{fig: desenho} and \ref{fig: desenho_3}). This concludes the proof.
\begin{figure}[htb]

\center
\includegraphics[scale=1.5]{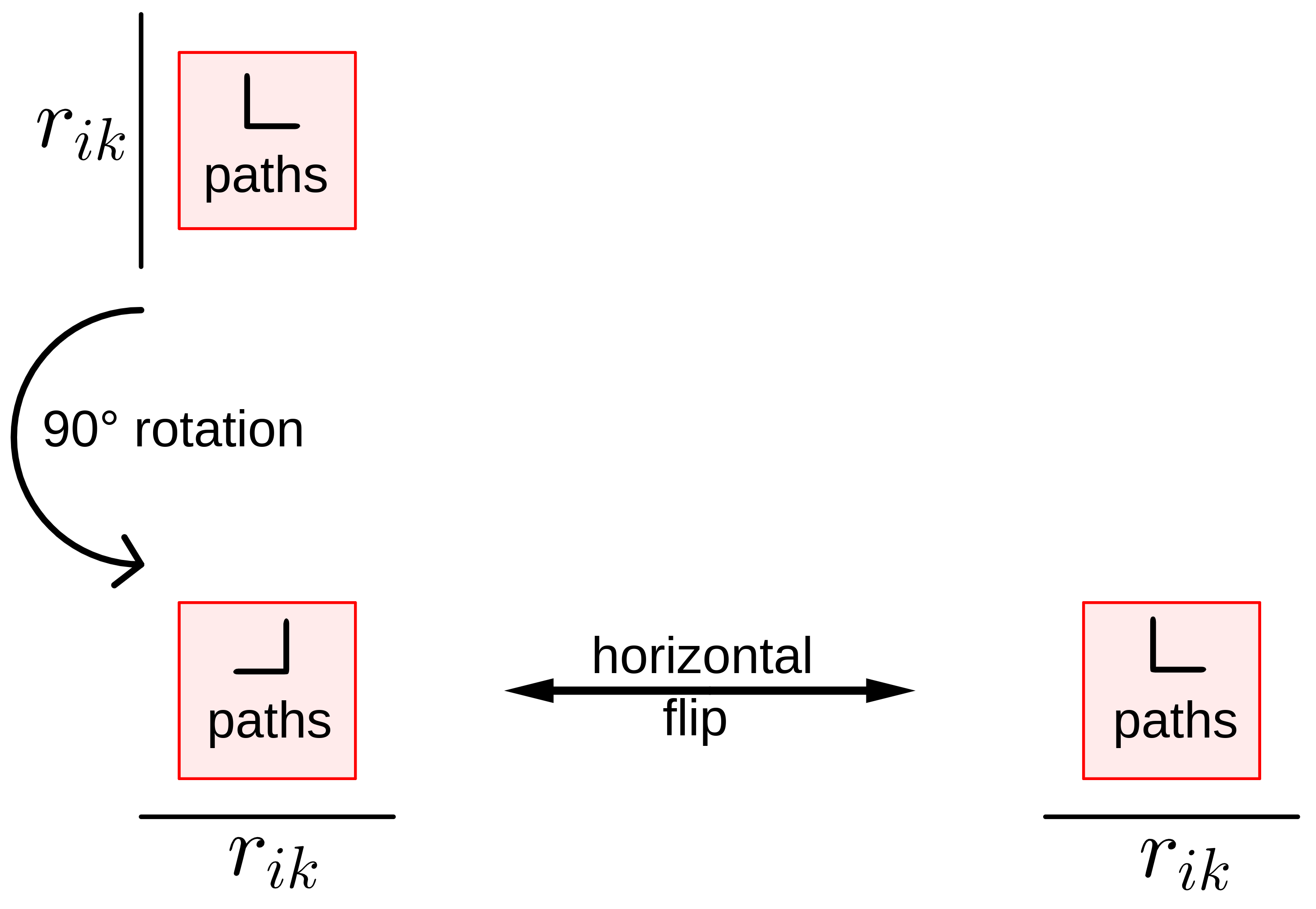}
\caption{Rotating $\llcorner$-EPG representations obtained after induction step.}
\label{fig: desenho_3}

\end{figure}
\end{proof}
A graph $G$ is a \emph{block graph} if every block of $G$ is a clique. In \cite{asinowski2012some}, the authors showed that block graphs are B$_1$-EPG in a proof by contradiction. As a consequence of Theorem~\ref{theorem1}, we have the following result.
\begin{cor}
Block graphs are $\llcorner$-EPG.
\end{cor}
It is known that trees are $\llcorner$-EPG. In \cite{golumbic2009edge}, the authors described a recursive procedure to construct an $\llcorner$-EPG representation of them. Note that trees are in particular block graphs and therefore can also be represented with the construction of Theorem~\ref{theorem2}. As an example, compare the resulting representations of both constructions, considering the tree $T$ in Figure~\ref{fig: Tree T}.
\begin{figure}[htb]
    \centering
    \tikzset{
    treenode/.style={align=center, inner sep=0pt},
    % Node style
    node_green/.style = {treenode, circle, black, font=\bfseries, draw=black, fill=green!30, text width=0.5cm}
}

\begin{tikzpicture}[level/.style={sibling distance=2cm, level distance=1.5cm}, scale=0.7]
    \tikzstyle{level 1} = [sibling distance=3.5cm]
    \node[node_green]{\mat{v_1}}
        child{node[node_green]{\mat{v_2}}
            child{node[node_green]{\mat{v_4}}}
            child{node[node_green]{\mat{v_5}}}
        }
        child{node[node_green]{\mat{v_3}}
            child{node[node_green]{\mat{v_6}}}
            child{node[node_green]{\mat{v_7}}
                child{node[node_green]{\mat{v_8}}}
                child{node[node_green]{\mat{v_9}}}
            }
        }
;
\end{tikzpicture}
    \caption{Tree $T$.}
    \label{fig: Tree T}
\end{figure}
The $\llcorner$-EPG representation of $T$ described in \cite{golumbic2009edge} is shown in Figure~\ref{fig: GOLUMBIC Tree T}. And the $\llcorner$-EPG representation of $T$ given by Theorem~\ref{theorem1} is shown in Figure~\ref{fig: VFJ Tree T}.

Finally, note that the induction proof of Theorem~\ref{theorem2} yields a recursive algorithm to produce an $\llcorner$-EPG representation given as input a graph $G$ holding the theorem conditions. This algorithm can be recognized in linear time, since the recognition of interval graphs (needed in order to obtain a model of each $B'_i$ defined in the theorem's proof) can be done in linear time \cite{Booth}.
\begin{figure}[htb]

\center
\subfigure[][]{\includegraphics[width=5cm]{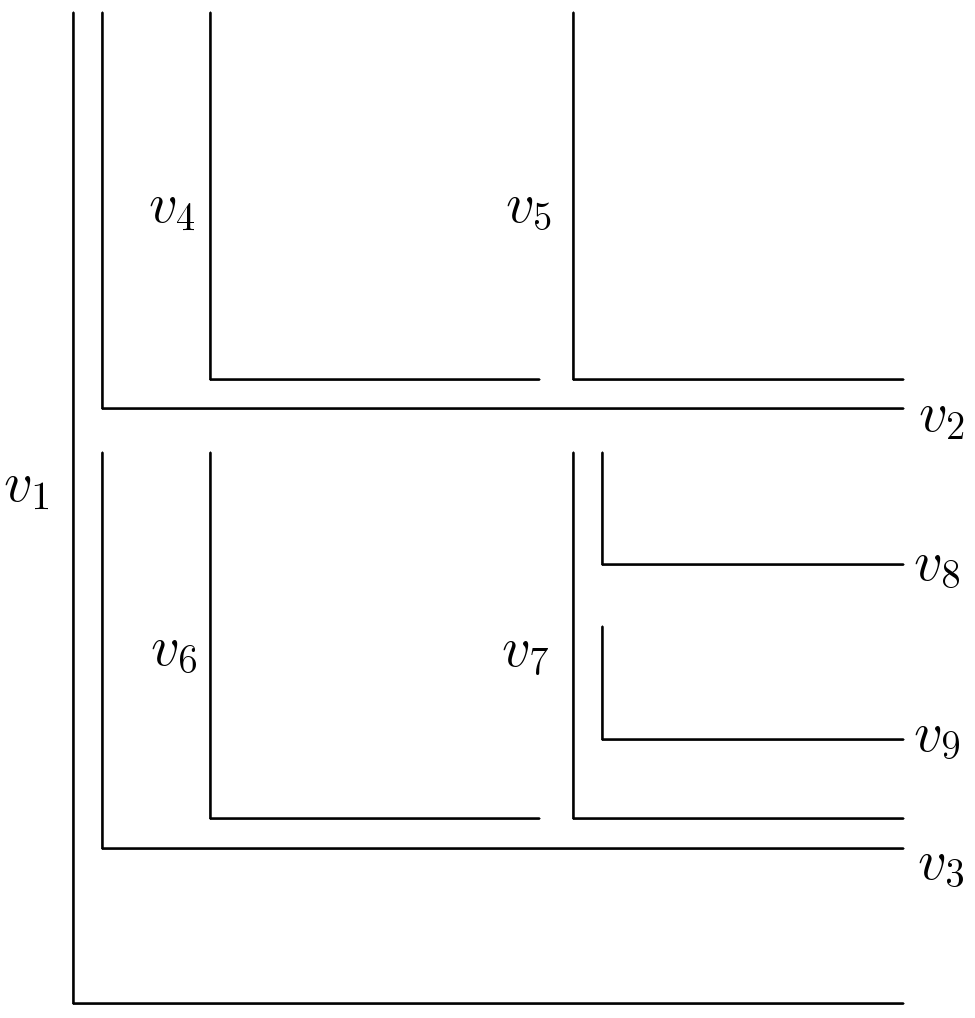} \label{fig: GOLUMBIC Tree T}}
\qquad
\subfigure[][]{\includegraphics[width=5cm]{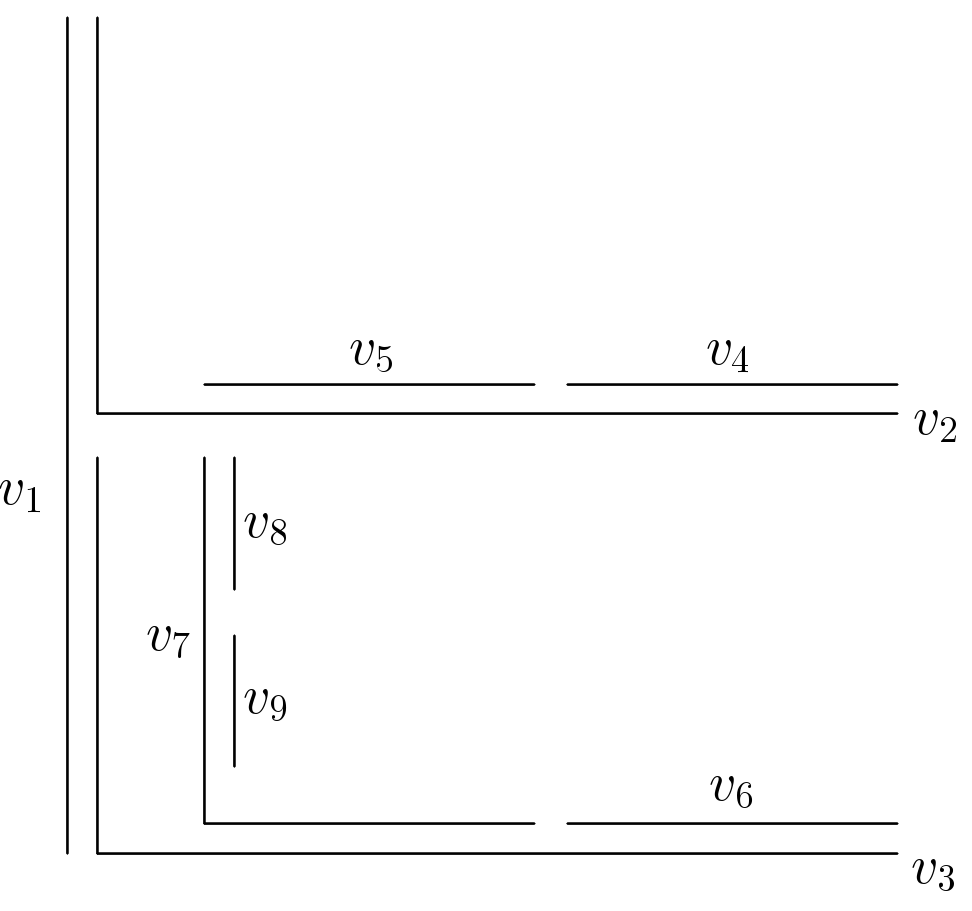} \label{fig: VFJ Tree T}}
\caption{(a) $\llcorner$-EPG representation of $T$. (b) Alternative $\llcorner$-EPG representation of $T$.}
\label{fig: representations of T}

\end{figure}

\section{A B$_{1}$-EPG representation of cactus graphs}\label{sec3}

A \emph{cactus graph} is a connected graph in which every block is either an edge or a cycle. So, cactus is another generalization of trees. In \cite{Cela}, the authors showed that cactus graphs are $\ulcorner\lrcorner$-EPG. In this section, we provide an alternative construction that yields B$_{1}$-EPG representations of cactus graphs using BC-trees.
\begin{thm}
Cactus graphs are B$_{1}$-EPG.
\end{thm}
\begin{proof}
Let $G$ be a cactus graph. This proof follows the same reasoning lines as those in the proof of Theorem~\ref{theorem1}. The key difference is that here the block vertices of a rooted BC-tree $T$ of $G$ represent the cycles (or single edges) in $G$. If $G$ has no cut vertices, then $G$ is either a single edge or a cycle, and the result is trivial. Assume, therefore, that $G$ has a cut vertex.

The theorem is proved by induction. Again, a stronger claim is actually proved, stated as follows: given any cactus graph $G$ and a BC-tree $T$ of $G$ rooted at some cut vertex $r$, there exists a B$_1$-EPG representation $\mathcal{R} = \{P_v \mid v \in V(G) \}$ of $G$ in which $P_r$ is a vertical path with no bends in $\mathcal{R}$.

\begin{figure}[htb]

\center
\subfigure[][]{\includegraphics[scale=0.5]{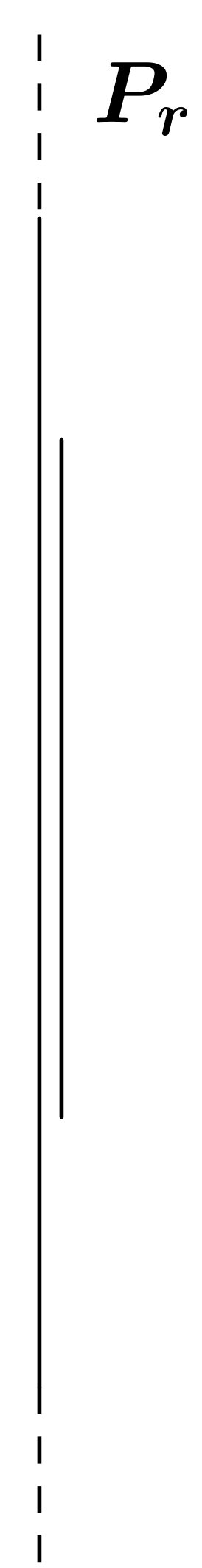} \label{fig: singleedge}}
\qquad
\subfigure[][]{\includegraphics[scale=0.5]{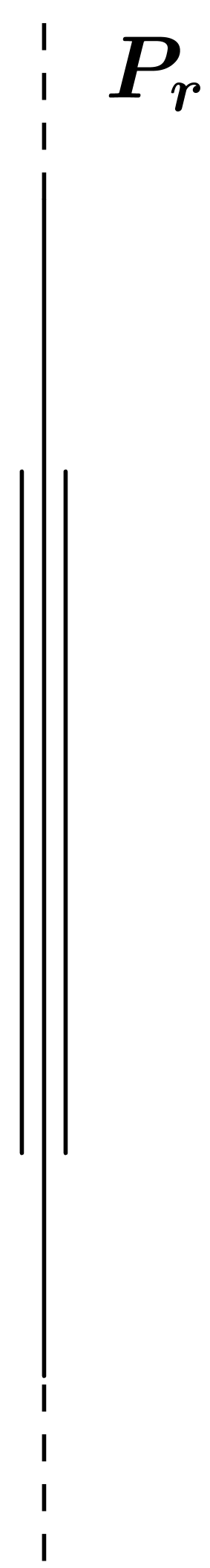} \label{fig: C3}}
\qquad
\subfigure[][]{\includegraphics[scale=0.5]{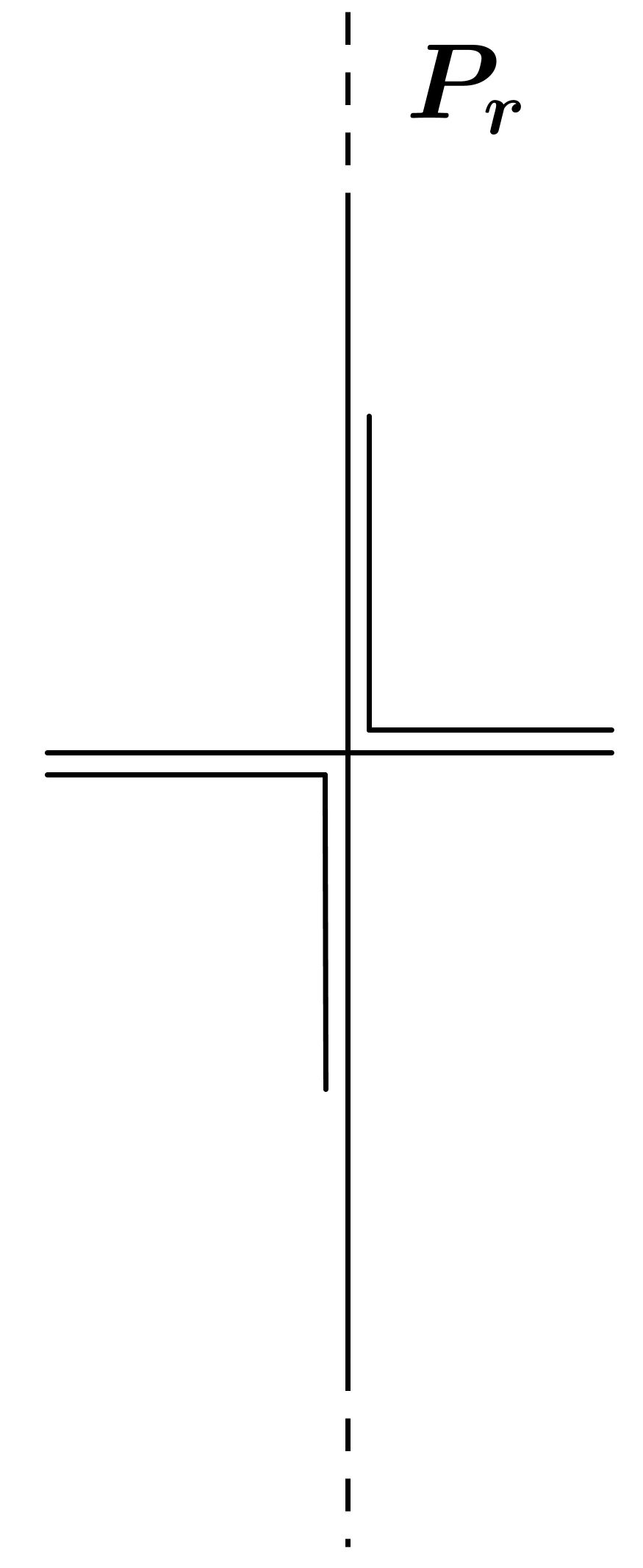} \label{fig: C4}}
\qquad
\subfigure[][]{\includegraphics[scale=0.5]{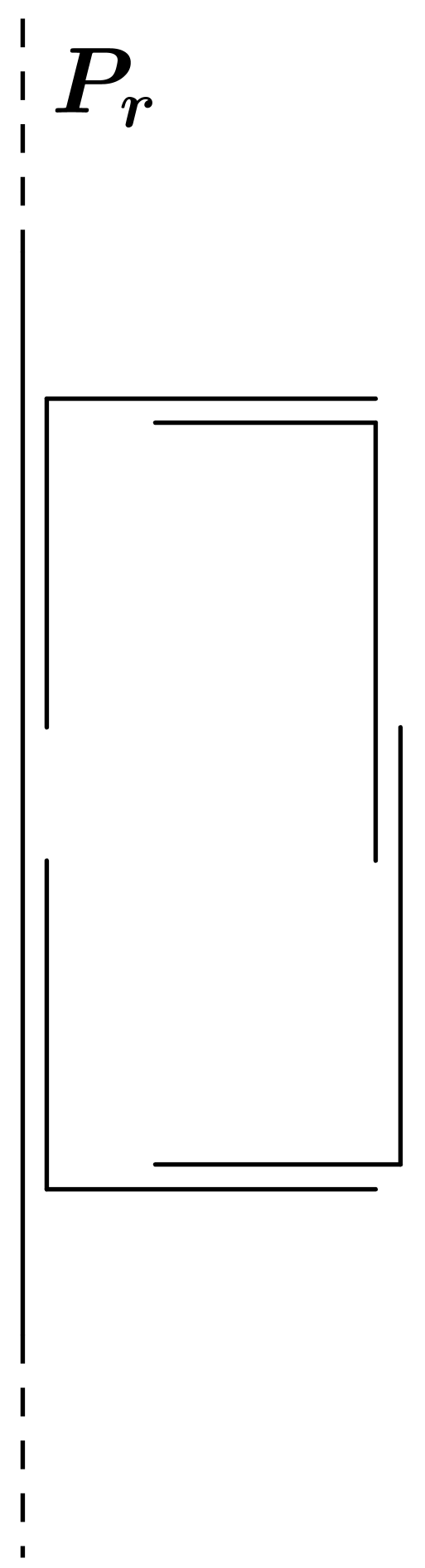} \label{fig: C5}}
\qquad
\subfigure[][]{\includegraphics[scale=0.5]{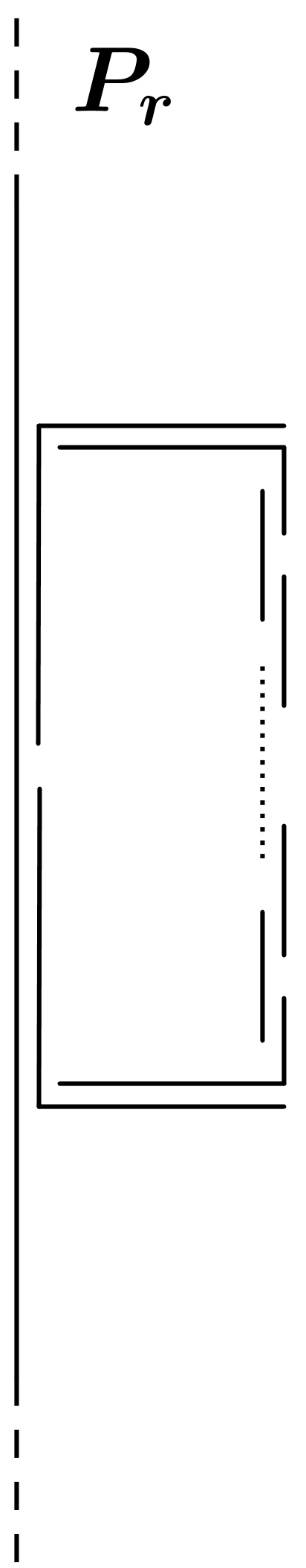} \label{fig: Ck}}
\caption{B$_1$-EPG representations of the blocks in $P_r$.}
\label{fig: blocks}

\end{figure}

Let $B_1, B_2, \ldots, B_t$ be the block vertices which are children of $r$ and let $T_{i1}, T_{i2}, \ldots, T_{ij_{i}}$ be the subtrees rooted at $B_i$ (see Figure~\ref{fig: BC-tree}), for all $1\leq i\leq t$.
\begin{figure}[htb]

\center
\subfigure[][]{\includegraphics[scale=0.6]{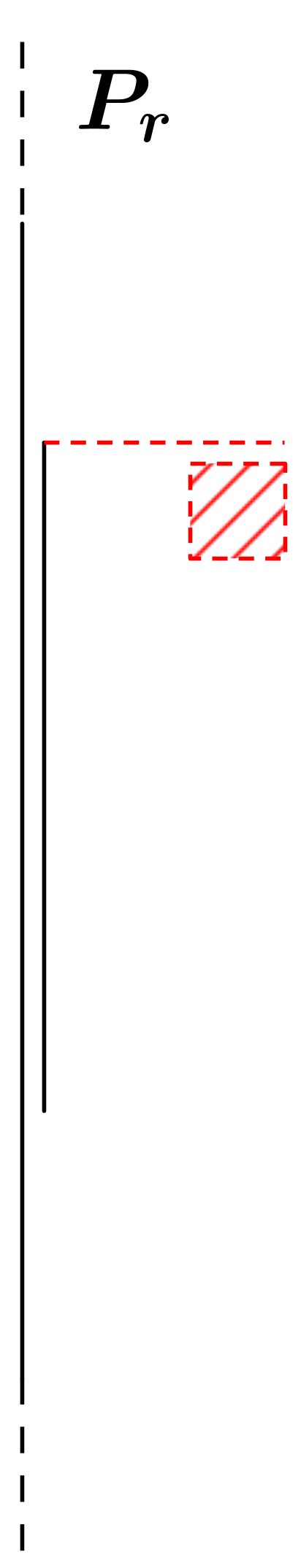} \label{fig: singleedge_2}}
\qquad
\subfigure[][]{\includegraphics[scale=0.6]{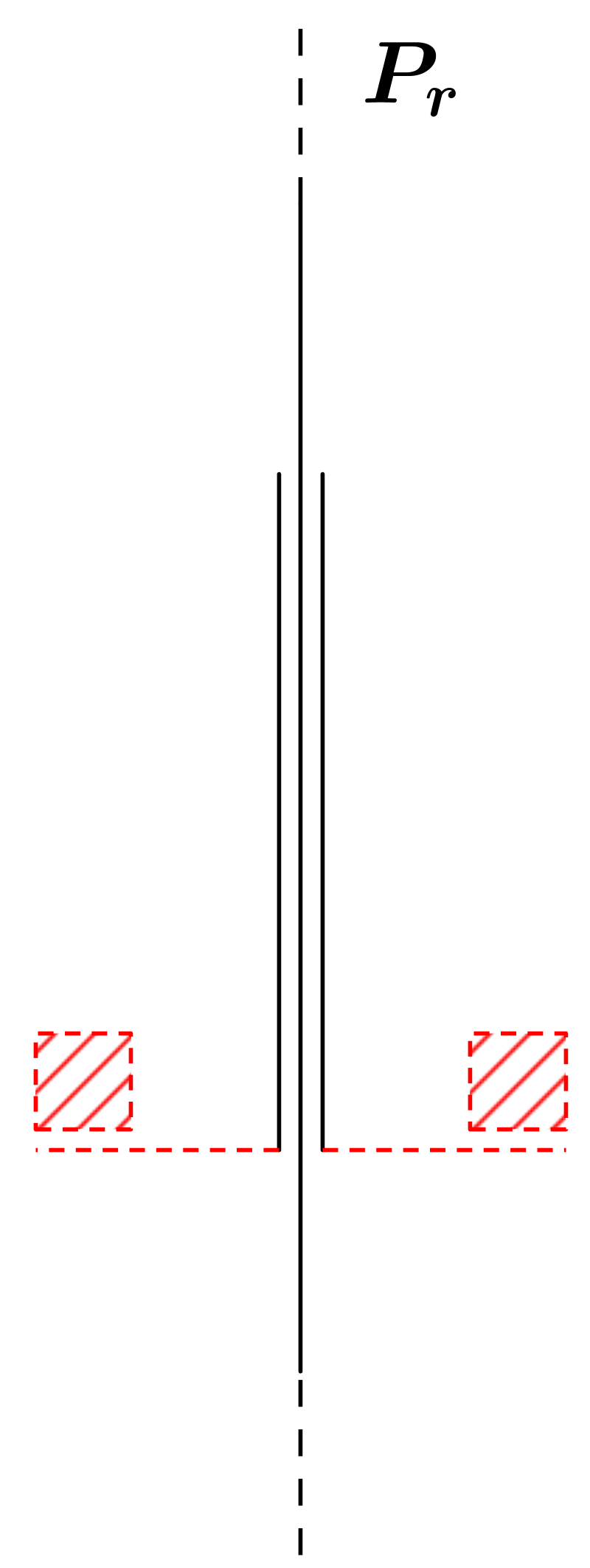} \label{fig: C3_2}}
\qquad
\subfigure[][]{\includegraphics[scale=0.5]{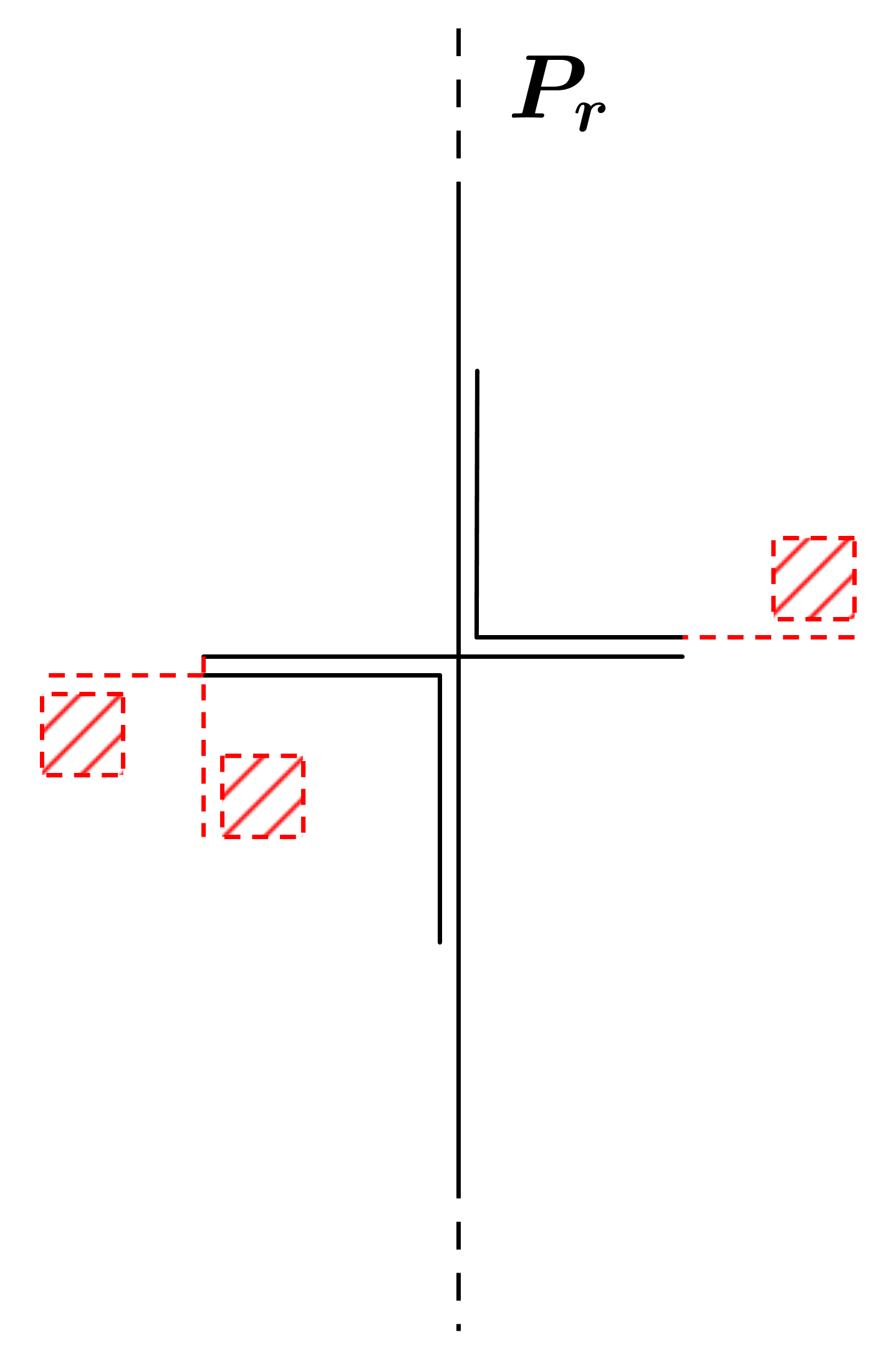} \label{fig: C4_2}}
\qquad
\subfigure[][]{\includegraphics[scale=0.5]{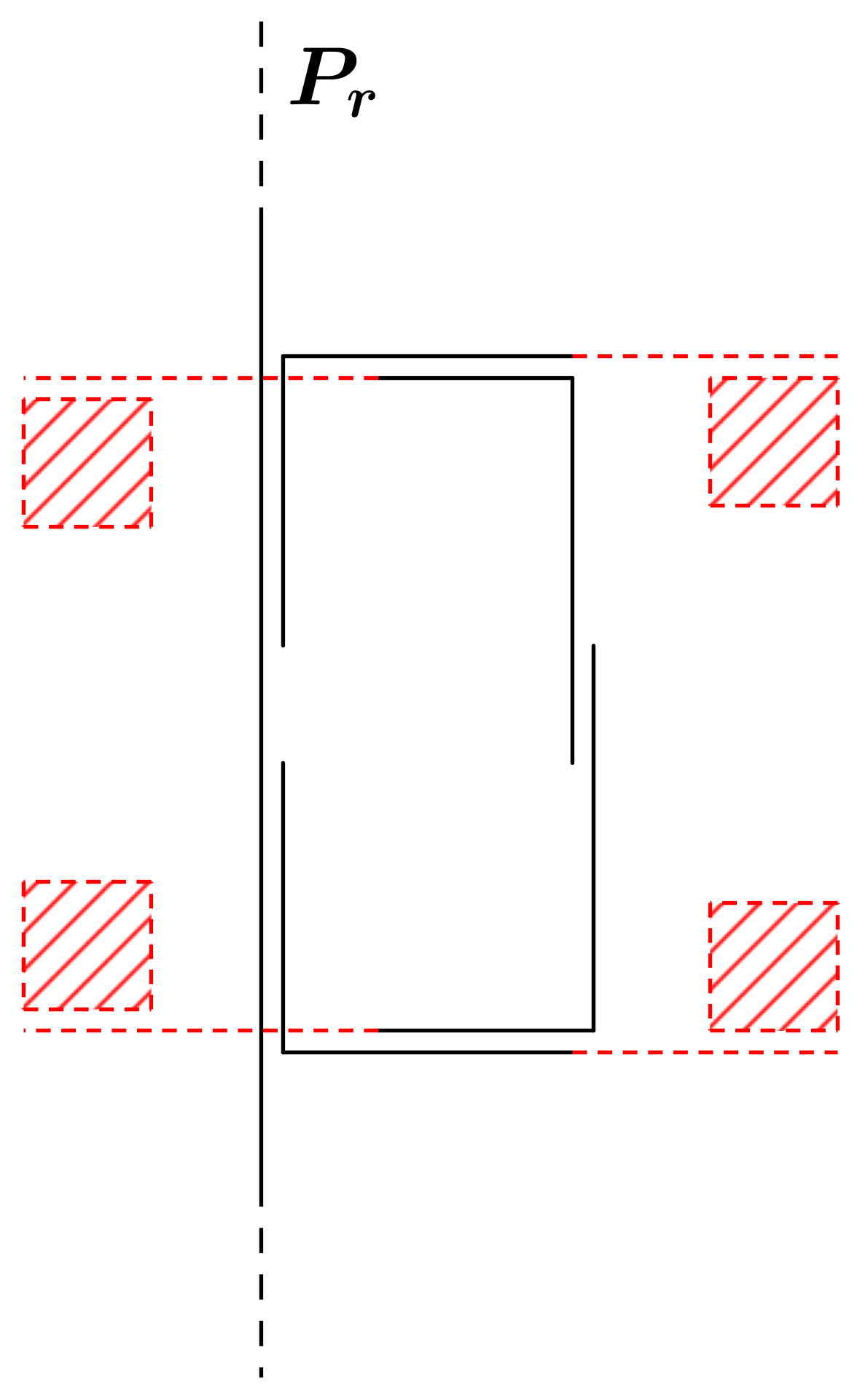} \label{fig: C5_2}}
\qquad
\subfigure[][]{\includegraphics[scale=0.6]{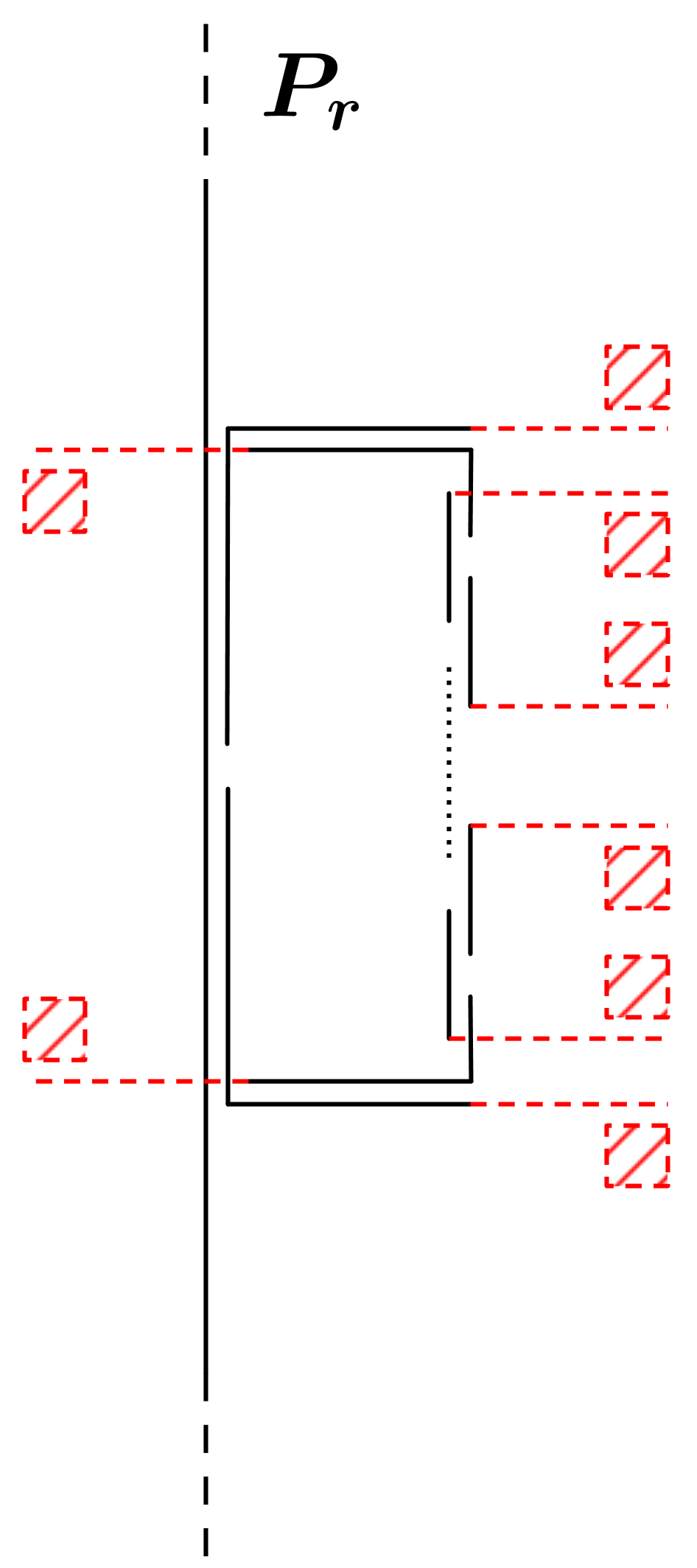} \label{fig: Ck_2}}
\caption{B$_1$-EPG representations (the red regions are to attach representations obtained from the hypothesis induction).}
\label{fig: Cases_2}

\end{figure}
From $T$, build the representation $\mathcal{R}$ of $G$ as follows. First, build an arbitrary vertical path $P_r$ in the grid $\mathcal{G}$, corresponding the root $r$. Next we must represent the blocks $B_i$, for all $1\leq i\leq t$, in which $r$ is contained. Let us consider the following cases:
\begin{itemize}
    \item[-] If $B_i$ is a single edge (resp. $C_3$), we can represent it as one path (resp. two paths) with no bends intersecting $P_r$ (see Figures~\ref{fig: singleedge} and \ref{fig: C3}).
    \item[-] If $B_i$ is a $C_4$, $C_5$, or $C_k$ with $k \geq 6$, the representation will be as shown in Figures~\ref{fig: C4}, \ref{fig: C5}, or \ref{fig: Ck}, respectively.
\end{itemize}
The blocks of $r$ can be placed along $P_r$. It is always possible to stretch the ends of $P_r$ to include all required blocks.

Using the same idea from the proof of Theorem~\ref{theorem1}, in order to recursively apply the same operation on each subtree rooted at the remaining cut vertices, we must guarantee that each one of them has a vertical or horizontal portion in which we will be able to attach the resulting B$_1$-EPG representation of their subtrees. Let $r_{ij}$ be the root of $T_{ij}$. Applying induction hypothesis, we obtain B$_1$-EPG representations of each subtree that have vertical paths representing each root. Note that, due to the constructions depicted in Figure~\ref{fig: blocks}, it is always possible to embed the B$_1$-EPG representations produced by the induction hypothesis in the regions depicted in Figure~\ref{fig: Cases_2}, so that the resulting representation is B$_1$-EPG.

\end{proof}
It is known that the class of outerplanar graphs, which is a superclass of cactus graphs, is B$_{2}$-EPG. It was conjectured by Biedl and Stern in \cite{Biedl} and proved by Heldt, Knauer and Ueckerdt in \cite{Heldt}. See an example of an outerplanar graph and its B$_{2}$-EPG representation in Figure~\ref{fig:outerplanar}.
\begin{figure}
    \centering
    \includegraphics[scale=1]{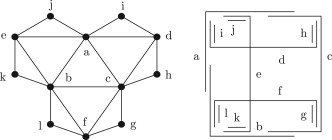}
    \caption{Outerplanar graph and its B$_2$-EPG representation (see \cite{Heldt}).}
    \label{fig:outerplanar}
\end{figure}

\section{Conclusion}
In this paper, we showed a B$_{1}$-EPG representation of graphs in which every block is B$_1$-EPG and every cut vertex is a universal vertex in the blocks to which it belongs. We extend the proof to show that cactus graphs are also B$_{1}$-EPG. We also showed a linear-time algorithm to construct $\llcorner$-EPG representations of graphs in which every block is $\llcorner$-EPG and every cut vertex is a universal vertex in the blocks to which it belongs, concluding that block graphs are $\llcorner$-EPG. The latter result provides an alternative $\llcorner$-EPG representation for trees.

\bibliographystyle{unsrt}

\end{document}